\documentclass[10pt]{amsart}
\usepackage{amsmath,amsthm, amsfonts,amssymb}
\date{}
\newtheorem{theorem}{Theorem}
\newtheorem{lemma}{Lemma}

\newtheorem{corollary}{Corollary}
\newtheorem{proposition}{Proposition}

\newtheorem{remark}{Remark}
\newcommand{\Cld}{\mathrm{Cld}}

\newcommand{\Comp}{\mathrm{Comp}}
\newcommand{\Fin}{\mathrm{Fin}}
\newcommand{\diam}{\mathrm{diam}}
\newcommand{\IN}{\mathbb{N}}
\newcommand{\cl}{\mathrm{cl}}
\newcommand{\dist}{\mathrm{dist}}
\newcommand{\w}{\omega}

\begin{document}
\title
[Attouch-Wets topology on hyoerspaces]{Hyperspaces  with the
Attouch-Wets topology \\ homeomorphic to $\ell _2$}
\author{R. Voytsitskyy}
\address{Department of Mathematics, Ivan Franko Lviv National
University, Universytetska 1, Lviv, 79000, Ukraine}
\email{voytsitski@mail.lviv.ua} 
\begin{abstract}
It is shown that the hyperspace of  all nonempty closed subsets
$\Cld_{AW}(X)$ of a separable metric space $(X,d)$ endowed with
the Attouch-Wets topology is homeomorphic to $\ell _2$ if and only
if the completion $\overline{X}$ of $X$ is proper, locally
connected and contains no  bounded connected component, $X$ is
topologically complete and not locally compact at infinity.
\end{abstract}


\maketitle 

\section {Introduction}

For a metric space $X =(X,d)$, let $C(X)$ be the set of all
continuous real valued functions on $X$ and $\Cld(X)$ be the set
of all nonempty closed subsets of $X$. Identifying each $A \in
\Cld(X)$ with the continuous function $X\ni x\mapsto d(x,A)\in
\mathbb R$, we can embed $\Cld(X)$ into the function space $C(X)$.

The function space $C(X)$ carries at least three natural
topologies: of point-wise convergence, of uniform convergence and
of uniform convergence on bounded subsets of $X$. Those three
topologies of $C(X)$ induce three topologies on the hyperspace
$\Cld(X)$: the \em Wijsman \em topology, the \em metric Hausdorff
\em topology and the \em Attouch-Wets \em topology. The hyperspace
$\Cld(X)$ endowed with one of these topologies is denoted by
$\Cld_W(X)$, $\Cld_H(X)$, and $\Cld_{AW}(X)$, respectively. The
Wijsman topology coincides with the Attouch-Wets topology if and
only if bounded subsets of $X$ are totally bounded
\cite[Theorem 3.1.4]{Be}. On the other hand, the Attouch-Wets
topology coincides with the Hausdorff metric topology if and only
if $(X,d)$ is a bounded metric space \cite[Exercise 3.2.2]{Be}.
The Hausdorff metric  topology on $\Cld_H(X)$ is generated by the
Hausdorff metric $d_H(A,B)=\sup_{x\in X}|d(x,A)-d(x,B)|$, where
$A,B\in\Cld(X)$.

In \cite[Theorem 5.3]{BKS} it is proved that for  an infinite-dimensional
Banach space $X$ of weight $w(X)$, the hyperspace $\Cld_{AW}(X)$
is homeomorphic to $(\cong)$ the Hilbert space of weight
$2^{w(X)}$. In particular, for  an
infinite-dimensional separable Banach space $X$, the hyperspace
$\Cld_{AW}(X)$ is homeomorphic to $\ell _{2}(2^{\aleph _{0}})$. On
the other hand, for each finite-dimensional normed linear space
$X$, since every bounded closed set in $X$ is compact, the
Attouch-Wets topology on $\Cld(X)$ agrees with the Fell topology
\cite [p.144]{Be}. Then, by \cite{SY},
$\Cld_{AW}(X)$ is homeomorphic to $Q\setminus \{pt\}$. Thus, for a Banach space $X$
the hyperspace $\Cld_{AW}(X)$ is either locally compact or
non-separable.
 In  \cite{BKS} the authors
asked: does there exist an unbounded metric space $X$ such that
$\Cld_{AW}(X)\cong \ell _2?$ And, more generally, what are the
necessary and sufficient conditions under which the hyperspace
$\Cld_{AW}(X)$ is homeomorphic to $\ell_2$? In this paper we answer these questions proving the following
characterization theorem.

\begin{theorem}\label{th1}
The hyperspace $\Cld_{AW}(X)$ of a metric space $X$ is
homeomorphic to $\ell _2$ if and only if the completion
$\overline{X}$ of $X$ is proper, locally connected and contains no
bounded connected component, $X$ is topologically complete  and is
not locally compact at infinity.
\end{theorem}

A metric space $X$ is defined to be
\begin{itemize}

\item {\em proper} if each closed bounded subset of $X$ is compact;

\item {\em not locally compact at infinity} if no bounded subset of
$X$ has locally compact complement.
\end{itemize}

Observe that under the  conditions of Theorem~\ref{th1} the
Attouch-Wets topology coincides with the Wijsman topology (cf.
\cite[Theorem 3.1.4]{Be}). So, for free, we obtain the following

\begin{corollary}\label{wijsman}
For a  metric space $(X,d)$  the hyperspace $\Cld_W(X)$ is
homeomorphic to $\ell_2$ if the completion $\overline{X}$ of $X$
is proper, locally connected and contains no bounded component,
$X$ is topologically complete and not locally compact at infinity.
\end{corollary}

Applying Theorem~\ref{th1} and Corollary~\ref{wijsman} to the
space $\mathbb P$  of irrational numbers  of the real line we
obtain:

\begin{corollary}
$\Cld_{AW}(\mathbb P)=\Cld_W(\mathbb P)\cong\ell _2$.
\end{corollary}

As a by-product of the proof of Theorem~\ref{th1} we obtain the
following characterization of metric spaces whose hyperspaces with
the Attouch-Wets topology are separable absolute retracts.

\begin{theorem}\label{th2} The hyperspace $\Cld_{AW}(X)$ of a
metric space $X$ is a separable absolute retract if and only
 if the completion $\overline X$ of $X$ is proper,
  locally connected and contains no bounded connected
 component.
\end{theorem}

Our Theorems~\ref{th1} and ~\ref{th2} are ``Attouch-Wets''
counterparts  of the following two results from  \cite{BV$_2$}:

\begin{theorem}\label{t1} The hyperspace $\Cld_H(X)$ of a metric
space $(X,d)$ is  homeomorphic to $\ell_2$ if and only if $X$ is a
topologically complete nowhere locally compact space and the
completion $\overline{X}$ of $X$ is compact, connected, and
locally connected.
\end{theorem}

\begin{theorem}\label{t2} The hyperspace $\Cld_H(X)$ of a  metric space $X$
 is a separable absolute retract
 if and only if the completion $\overline X$ of $X$ is compact,
  connected and locally connected.
\end{theorem}

\section{Topology of Lawson semilattices and some auxiliary facts}

Theorem~\ref{th2} will be derived from a more general result
concerning Lawson semilattices.
 By a {\em topological semilattice} we understand
  a pair $(L,\vee)$ consisting of a topological space $L$ and a continuous
   associative commutative idempotent operation $\vee:L\times L\to L$.
    A topological semilattice $(L,\vee)$ is a {\em Lawson semilattice} if open
     subsemilattices form a base of the topology of $L$.
A typical example of a Lawson semilattice is the hyperspace
$\Cld_H(X)$
 endowed with the operation of union $\cup$.

 Each semilattice $(L,\vee)$ carries a natural partial order: $x\le y$ iff $x\vee y=y$.
A semilattice $(L,\vee)$ is called {\em complete} if each subset
$A\subset L$ has the smallest upper bound $\sup A\in L$. It is
well-known (and can be easily proved) that each compact
topological semilattice is complete.

\begin{lemma}\label{lowloc} If $L$ is a locally compact Lawson semilattice, then
 each compact subset $K\subset L$ has the smallest upper bound $\sup K\in L$.
  Moreover, the map $\sup:\Comp(L)\to L$, $\sup:K\mapsto\sup K$, is a continuous
   semilattice homomorphism. Also for every subset $A\subset L$ with compact
   closure $\overline A$
    we have $\sup A=\sup\overline A$.
\end{lemma}

This lemma easily follows from its compact version proved by J.
Lawson in \cite{Law}.

In Lawson semilattices many geometric questions reduce to the
one-dimensional level. The following fact illustrating this
phenomenon is proved in \cite{KSY}.

\begin{lemma}\label{lawson}
Let $X$ be a dense subsemilattice of a metrizable Lawson
semilattice $L$.
 If $X$ is relatively $LC^0$ in $L$ (and $X$ is
path-connected), then $X$ and $L$ are ANRs (ARs) and $X$ is
homotopy dense in $L$.
\end{lemma}

A subset $Y\subset X$ is defined to be {\em relatively $LC^0$} in
$X$ if for every $x\in X$ each neighborhood $U$ of $x$  in $X$
contains a smaller neighborhood $V$ of $x$ such that every two
points of $V\cap Y$ can be joined by a path in $U\cap Y$.

Under a suitable completeness condition, the density of a
subsemilattice
 is equivalent to the homotopical density. A subset $Y$ of a
 topological space $X$ is {\em homotopy dense} in $X$ if there is
 a homotopy $(h_t)_{t\in[0,1]} : X\rightarrow X$ such that
 $h_0=id$ and $h_t(X)\subset Y$ for every $t>0$.

A subsemilattice $X$ of semilattice $L$ is defined to be {\em
relatively complete}
 in $L$ if for any subset $A\subset X$ having the smallest upper
  bound $\sup A$ in $L$ this bound belongs to $X$.

\begin{proposition}\label{complete} Let $L$ be a locally compact locally connected
 Lawson semilattice. Each dense relatively complete subsemilattice $X\subset L$ is homotopy
  dense in $L$.
\end{proposition}

\begin{proof} According to Lemma~\ref{lawson} it suffices to check
that $X$ is relatively $LC^0$ in $L$. Given a point $x_0\in L$ and
a neighborhood $U\subset L$ of $x_0$,  consider the canonical
retraction $\sup:\Comp(L)\to L$. Using the ANR-property of $L$ and
continuity of $\sup$, find a path-connected neighborhood $V\subset
L$ of $x_0$ such that $\sup(\Comp(\overline{V}))\subset U$. We
claim that any two points $x,y\in X\cap V$ can be connected by a
path in $X\cap U$. First we construct a path $\gamma:[0,1]\to
\overline{V}$ such that $\gamma(0)=x$, $\gamma(1)=y$ and
$\gamma^{-1}(X)$ is dense in $[0,1]$. Let $\{q_n:n\in\w\}$ be a
countable dense subset in $[0,1]$ with $q_0=0$ and $q_1=1$. The
space $L$, being locally compact, admits a complete metric $\rho$.
The path-connectedness of $V$ implies the existence of a
continuous map $\gamma_0:[0,1]\to V$ such that $\gamma_0(0)=x$ and
$\gamma_0(1)=y$. Using the local path-connectedness of $L$ we can
construct inductively a sequence of functions $\gamma_n:[0,1]\to
V$ such that
\begin{itemize}
\item $\gamma_n(q_k)=\gamma_{n-1}(q_k)$ for all $k\le n$;
\item $\gamma_n(q_{n+1})\in X$;
\item $\sup_{t\in[0,1]}\rho(\gamma_n(t),\gamma_{n-1}(t))<2^{-n}$.
\end{itemize}

Then the map $\gamma=\lim_{n\to\infty}\gamma_n:[0,1]\to
\overline{V}$ is continuous and has the desired properties:
$\gamma(0)=x$, $\gamma(1)=y$ and $\gamma(q_n)\in X$ for all
$n\in\w$.

For every $t\in[0,1]$ consider the set
$\Gamma(t)=\{\gamma(s):|t-s|\le \dist(t,\{0,1\})\}$. It is clear
that the map $\Gamma:[0,1]\to\Comp(L)$ is continuous and so is the
composition $\sup\circ\Gamma:[0,1]\to L$. Observe that $\sup\circ
\Gamma(0)=\sup\{\gamma(0)\}=\gamma(0)=x$, $\sup\circ\Gamma(1)=y$,
and $\sup\circ\Gamma([0,1])\subset\sup(\Comp(\overline{V}))\subset
U$. Since for every $t\in(0,1)$ the set
$\Gamma(t)=\overline{\Gamma(t)\cap X}$, we get
$\sup\Gamma(t)=\sup(\Gamma(t)\cap X)\in X$ by the relative
completeness of $X$ in $L$. Thus $\sup\circ\Gamma:[0,1]\to U\cap
X$ is a path connecting $x$ and $y$ in $U$.
\end{proof}
For a metric space $X$ by $\Fin(X)$ we denote the subspace of
$\Comp(X)$
 consisting of non-empty finite subspaces of $X$.

\begin{lemma}\label{l6} If $Y$ is a subset of a locally path-connected space $X$,
 then the subset $L=\Fin(X)\setminus \Fin(Y)$ is relatively $LC^0$ in $\Comp(X)$.
\end{lemma}

\begin{proof} By the argument of \cite{CN} we can show that $\Fin(X)$ is
relatively $LC^0$ in $\Comp(X)$. Consequently, for every compact
set $K\in\Comp(X)$ and a neighborhood $U\subset\Comp(X)$ of $K$
there is a neighborhood $V\subset\Comp(X)$ of $K$ such that any
two points $A,B\in\Fin(X)\cap V$ can be linked by a path in
$\Fin(X)\cap U$. Since $\Comp(X)$ is a Lawson semilattice, we may
assume that $U$ and $V$ are subsemilattices of $\Comp(X)$. We
claim that any two points $A,B\in L\cap V$ can be connected by a
path in $L\cap U$. Since $L\subset\Fin(X)$, there is a path
$\gamma:[0,1]\to U\cap \Fin(X)$ such that $\gamma(0)=A$ and
$\gamma(1)=B$. Define a new path $\gamma':[0,1]\to U\cap \Fin(X)$
letting $\gamma'(t)= \gamma(\max\{0,2t-1\})\cup
\gamma(\min\{2t,1\})$. Observe that $A\subset\gamma'(t)$ if $t\le
1/2$ and $B\subset\gamma'(t)$ if $t\ge 1/2$. Since $A,B\notin
\Fin(Y)$, we conclude that $\gamma'([0,1])\subset L\cap U$.
\end{proof}

We also need the following nontrivial fact from \cite[Corollary
2]{BV$_1$}.
\begin{lemma}\label{dense}
Let $X$ be a dense subset of a metric space $M$. Then the
hyperspace $\Cld_H(X)$ is an A(N)R if and only if so is the
hyperspace $\Cld_H(M)$.
\end{lemma}

 The proof of
Theorem~\ref{th1} and Theorem~\ref{t2} relies on the next lemma
due to D. Curtis \cite{Cu}.
\begin{lemma}\label{tool}
A homotopy dense $G_\delta$-subset $X\subset Q$ with
 homotopy dense complement in the Hilbert cube $Q$ is homeomorphic to $\ell_2$.
\end{lemma}

\section {The metrics $d_{AW}$ and $d_{H}$ on $\Cld(X)$}

Let $X=(X,d)$ be a metric space. The $\varepsilon$-neighborhood of
$x\in X$  (i.e., the open ball centered at $x$ with radius
$\varepsilon$) is denoted by $B(x, \varepsilon)$. Let $A$ and $B$
be nonempty subsets of a metric space $(X,d)$. The \em excess \em
of $A$ over $B$ with respect to $d$ is defined by the formula
$$e_d(A,B) = \sup \{d(a,B)\;|\;a\in A\}.$$ Here we assume that
$e_d(A,\emptyset)=+\infty$.
For the Hausdorff metric we have the following : $$d_H (A,B) =
\max \{ e_d(A,B), e_d(B,A) \} = \sup_{x\in X}|d(x,A)-d(x,B)|.$$

Now we define the metric $d_{AW}$.

Fix $x_0 \in X$ and let $X_i = \{x\in X \;|\; d(x,x_0)\leq i\}$.
The following metric $d_{AW}$
 on $\Cld(X)$ generates
the Attouch-Wets topology:\footnote{In \cite{Be}, the following
metric is adopted $$ d_{AW}(A,B)=\sum_{i\in \IN} 2^{-i} \min
\Bigl\lbrace 1,\; \sup_{x\in X_i} |d(x,A)-d(x,B)|\Bigr\rbrace. $$}
$$d_{AW}(A,B)=\sup_{i \in \IN} \min \Bigl \{1/i,\; \sup_{x\in X_i}
|d(x,A)-d(x,B)| \Bigr\}. $$ It should be noticed that
$$d_{AW}(A,B)\leq d_{H}(A,B) \;\; \mbox{  for every } A, B\in \Cld(X). $$

We  need the following fact for the Attouch-Wets convergence in
terms of excess, see \cite[Theorem 3.1.7]{Be}.
\begin{proposition}\label{p1}
Let $(X,d)$ be a metric space, and $A,A_1,A_2,...$ be nonempty
closed subsets of $X$,  $x_0\in X$ be fixed. The following are
equivalent:
\begin{enumerate}
\item $\lim_{n\rightarrow \infty}d_{AW}(A_n,A)=0$;
\item For each $i\in\mathbb N$, we have both $\lim_{n\rightarrow \infty }e_d(A\cap X_i,A_n)=0$
and $\lim_{n\rightarrow \infty }e_d(A_n\cap X_i,A)=0$.
\end{enumerate}
\end{proposition}

Recall that the Attouch-Wets topology depends on the metric for
$X$, that is, the space $\Cld_{AW}(X)$ is not a topological
invariant for $X$. Concerning conditions that two metrics  for $X$
induce the same topology, see \cite[Theorem 3.3.3]{Be}.

\section{Embedding $\Cld_{AW}(X)$ in $\Cld_{H}(\alpha X)$}
The main idea in proving Theorems~\ref {th1}, ~\ref{th2} is the
following: we reduce the Attouch-Wets topology on $\Cld(X)$ to the
Hausdorff metric topology on $\Cld(\alpha X)$ for a suitable
one-point extension $\alpha X$ of $X$. The metric space $(\alpha
X, \rho )$ is obtained by adding the infinity point $\infty$ to
the space $X$. More precisely, we endow the space $\alpha X$ with
the metric
$$\rho (x,y)=\begin{cases}
\min \Bigl\{d(x,y),\frac{1}{1+d(x,x_0)}+\frac{1}{1+d(y,x_0)}\Bigr\}
, & \mbox{if  } x, y\in X \\
 \frac{1}{1+d(x,x_0)}, &\mbox{if } x\in X, \;y=\infty\\
\frac{1}{1+d(y,x_0)}, &\mbox{if } y\in X, \;x=\infty\\
0,&\mbox{if }x=y=\infty.
\end{cases} $$
Here, $x_0 \in X$ is a fixed point. Note, that $(X,d)$ is
homeomorphic to $(\alpha X\setminus\{\infty\}, \rho)$ and
$\diam(\alpha X) < 2$.

\begin{remark}
We can obtain the space $(\alpha X,\rho)$ in the following way:
 embed $X$ in $X\times [0,1)$ by the
formula $x\mapsto (x,\frac{d(x,x_0)}{1+d(x,x_0)})$ and consider
the cone metric on this space (induced by the suitable metrization
of the quotient space $X\times [0,1]/X\times \{1\}$).
\end{remark}

\begin{proposition}\label{embed}
The function $e: \Cld_{AW}(X)\rightarrow \Cld_{H}(\alpha X)$
defined by the formula $e(A)=A\cup\{\infty\}$ is an embedding.
\end{proposition}
\begin{proof}
Let $\lim_{n\rightarrow \infty} d_{AW}(A_n,A)=0$. Assume to the
contrary that $\lim_{n\rightarrow \infty}
\rho_{H}(e(A_n),e(A))\not=0$. This means that there exists some
$\varepsilon _0>0$ such that we can find either a sequence
$x_{n_k}\in A_{n_k}$, $k \in \mathbb N$, with
$\rho_{H}(x_{n_k},A\cup \{\infty\})\geq \varepsilon_0$, or there
exists a sequence $(y_k)\subset A$ with
$\rho_{H}(A_{n_k}\cup\{\infty\},y_k)\geq\varepsilon _0$ for all $k
\in \IN$.\\ In the former case, since $\infty\in e(A)$, we have
$\rho(x_{n_k},\infty)=\frac{1}{1+d(x_{n_k},x_0)}\geq\varepsilon_0$
for each $k\in\IN$. Hence, there exists some $i_0\in \mathbb N$
with $(x_{n_k})\subset X_{i_0}$, $k\in\IN$. For every $y\in A$ and
$k\in\IN$  $\rho(x_{n_k},y)\geq\varepsilon _0$, and so
$d(x_{n_k},y)\geq\varepsilon _0$. This implies that for each
$k\in\IN$ $\sup_{x\in X_{i_0}} |d(x,A_{n_k})-d(x,A)|\geq
d(x_{n_k},A)\geq\varepsilon _0$. Combining this with the
definition of the Attouch-Wets metric we get
$d_{AW}(A_{n_k},A)\geq\min \{\frac{1}{i_0}, \varepsilon_0\}$ for
all $k\in\IN$. This is  a contradiction. In the latter case,
similar to  the above, we have for each $k\in\IN$ $ \sup_{x\in
X_{i_0}} |d(x,A_{n_k})-d(x,A)|\geq d(y_k,A_{n_k})\geq\varepsilon
_0$.
 Whence,
$d_{AW}(A_{n_k},A)\geq\min \{\frac{1}{i_0}, \varepsilon_0\}$ for
all $k\in \mathbb N$.

\bigskip
Conversely, let $\lim_{n\rightarrow \infty}
\rho_{H}(e(A_n),e(A))=0$. Assume to the contrary that
$\lim_{n\rightarrow \infty}d_{AW}(A_n,A)\not=0.$ This means that
there exists a subsequence $(A_{n_k})\subset (A_n)$ with
$d_{AW}(A_{n_k},A)\geq\varepsilon _0 $ for some $\varepsilon
_0>0$. Then, by Proposition~\ref{p1}, there exists $i_0\in \mathbb
N$ such that either $e_d(A_{n_k}\cap X_{i_0},A)\geq\varepsilon _0$
 or $e_d(A\cap X_{i_0},A_{n_k})\geq\varepsilon _0$. Remark, that we can take
$i_0$ so large that $A_n\cap X_{i_0}\not=\emptyset$ for all
$n\in\IN$. Consequently, in the former case we can find a sequence
$x_k\in A_{n_k}\cap X_{i_0}$ with $d(x_k,y)\geq\varepsilon_0$ for
each $y\in A$ and $k\in\IN$. Hence, $\rho (x_k,y)=
\min\{d(x_k,y),\frac{1}{1+d(x_k,x_0)} + \frac
{1}{1+d(y,x_0)}\}\geq \min\{\varepsilon_0, \frac{1}{1+i_0}\}$ for
all $y\in A$ and $k\in\IN$. Since
$\rho(x_k,\infty)=\frac{1}{1+d(x_k,x_0)}\geq \frac{1}{1+i_0}$, it
follows that $\lim_{k\rightarrow \infty
}\rho_H(A_{n_k}\cup\{\infty\},
A\cup\{\infty\})\geq\min\{\varepsilon_0, \frac{1}{1+i_0}\}>0$, and
we have a contradiction. In the latter case, there exists a
sequence $y_k\in A\cap X_{i_0}$ with $d(y_k,x)\geq\varepsilon_0$
for all $x\in A_{n_k}$ and $k\in\mathbb N$. Whence, for every
$k\in\mathbb N$ we have $\rho(y_k,x)\geq \min\{\varepsilon_0,
\frac{1}{1+d(y_k,x_0)}\}\geq \min\{\varepsilon_0,
\frac{1}{1+i_0}\}$ and
$\rho(y_k,\infty)=\frac{1}{1+d(y_k,x_0)}\geq \frac{1}{1+i_0}$.
This violates that
$\lim_{n\rightarrow\infty}\rho_{H}(e(A_n),e(A))=0$.
\end{proof}

\medskip

For a metric space $X$ and a point $x_0\in X$ let
$$\Cld_H(X|\{x_0\})=\{F\in\Cld_H(X):x_0\in F\}.$$

It follows from Proposition~\ref{embed} that $$e(\Cld_{AW}(X))=\Cld_H(\alpha X|\{\infty\})\setminus\{\infty\}$$
 and thus $\Cld_{AW}(X)$ is homeomorphic to $\Cld_H(\alpha X|\{\infty\})\setminus\{\infty\}$.

The ANR-property of the space $\Cld_H(X|\{x_0\})$ was characterized in \cite{Voy}:

\begin{proposition}\label{denseat} For a metric space $X$ with a distinguished point $x_0$
 the hyperspace $\Cld_H(X|\{x_0\})$ is an absolute (neighborhood) retract if and only if
  so is the hyperspace $\Cld_H(X)$.
\end{proposition}

Proposition~\ref{denseat} implies the following fact about
Attouch-Wets hyperspace topology having an independent interest.

\begin{corollary}\label{completion}
Let $X$ be a dense subset of  a metric space $M$. Then, the
hyperspace $\Cld_{AW}(X)$ is an absolute neighborhood retract (an
absolute retract) if and only if so is the hyperspace
$\Cld_{AW}(M)$.
\end{corollary}

\begin{proof}
It follows from the Propositions~\ref{embed},
\ref{denseat} and Lemma~\ref{dense}.
\end{proof}

\section{The completion $\overline{\alpha X}$ is a Peano continuum}
 We need the following lemma proved in \cite[Lemma 2]{SY}.
\begin{lemma}\label{SY}
If $X$ is a locally connected, locally compact separable
metrizable space with no compact components, then its Alexandroff
one-point compactification  $\alpha X$  is a Peano continuum.
\end{lemma}
Using the previous lemma we can easily obtain
\begin{lemma}\label{key} Suppose that the completion $\overline{X}$ of a metric space $X$
is a proper locally connected space with  no bounded connected
components. Then, $\overline{\alpha X}$ is a Peano continuum.
\end{lemma}

\begin{proof}
Note, that the completion $\overline{X}$ of $X$ satisfies the
conditions of Lemma~\ref{SY} and $\overline{\alpha X}$ (the
completion of $\alpha X$) coincides with the Alexandroff one-point
compactification of $\overline{X}$.
\end{proof}

Then, by the  Curtis-Schori Hyperspace Theorem  \cite{CS},
$\Cld_H(\overline{\alpha X})=\Comp(\overline{\alpha X})$ is
homeomorphic to the Hilbert cube $Q$.

\begin{lemma}\label{cubinfty}
If a metric space $X$ has proper, locally connected completion $\overline{X}$
 having no bounded connected component, then the
hyperspace $\Cld_H(\overline{\alpha X}|\{\infty\})$ is homeomorphic to
the Hilbert cube $Q$.
\end{lemma}

\begin{proof}
Observe that $\Cld_H(\overline{\alpha X}|\{\infty\})$ is a retract of
$\Cld_H(\overline{\alpha X})$, and thus is a compact absolute
retract. Then, we use the Characterization Theorem for the Hilbert
cube, see \cite[Theorem 1.1.23]{BRZ}. By this theorem we have to
check that for each $\varepsilon >0$, every $n\in\IN$, and each
maps $f_1, f_2 : I^n \rightarrow \Cld_H(\overline{\alpha
X}|\{\infty\})$ there are maps $f_1', f_2' : I^n \rightarrow
\Cld_H(\overline{\alpha X}|\{\infty\})$ such that $d(f_i,
f_i')<\varepsilon$, $i=1,2$, and $f_1'(I^n)\cap
f_2'(I^n)=\emptyset$. Fix $\varepsilon>0$, $n\in\IN$, and maps
$f_1, f_2 : I^n \rightarrow \Cld_H(\overline{\alpha X}|\{\infty\})$.
By the argument of \cite{CN} we can show that
$\Fin_H(\overline{\alpha X})$ is homotopy dense in
$\Cld_H(\overline{\alpha X})$. Therefore, we can find an
$\varepsilon/2$-close to $f_i$ map
$g_i:I^n\rightarrow\Fin_H(\overline{\alpha X})$, $i=1,2$,
respectively, see \cite[Ex. 1.2.10]{BRZ}. Observe, that
$d(f_i,g_i\cup\{\infty\})<\varepsilon/2$, $i=1,2$. Then, it is
easily seen that maps $f_1'=g_1\cup\{\infty\}$ and $f_2'=g_2\cup
B(\infty,\varepsilon/2)$ are as required.
\end{proof}

\section{Proof of Theorem 2}
To prove the ``only if'' part, assume that $\Cld_{AW}(X)$  is a
separable absolute retract. The separability of $\Cld_{AW}(X)$
implies that each bounded subset of $X$ is totally bounded
\cite[Theorem 5.2]{BKS}, which is equivalent to the properness of
the completion $\overline{X}$ of $X$. By
Corollary~\ref{completion}, the hyperspace
$\Cld_{AW}(\overline{X})$  is a separable absolute retract too. In
this case $\Cld_{AW}(\overline{X})=\Cld_F(\overline{X})$ (by
$\Cld_F(X)$ we denote the hyperspace $\Cld(X)$ endowed with the
Fell topology, see \cite[Theorem 5.1.10 ]{Be}) is an absolute
retract, and we can apply \cite[Propositions 1, 2]{SY} to conclude
that the locally compact space $\overline{X}$ is locally connected
and contains no bounded (=compact) connected component.

Next, we prove the ``if'' part of Theorem~\ref{th2}. Assume that
the completion $\overline{X}$ of $X$ is proper, locally connected
with no bounded connected components. By Lemma~\ref{cubinfty}, the space
  $\Cld_H(\overline{\alpha X}|\{\infty\})$ is homeomorphic to the Hilbert cube $Q$.
Proposition~\ref{complete} implies that $\Cld_H(\alpha
X|\{\infty\})$ is
 homotopy dense in $\Cld_H(\overline{\alpha X}|\{\infty\})$.
  Taking into account that the Hilbert cube with deleted point is an absolute retract
   and so is any homotopy dense subset of $Q\setminus\{pt\}$, we
    conclude that $\Cld_H(\alpha X|\{\infty\})\setminus\{\infty\}$ and its
     topological copy $\Cld_{AW}(X)$ are absolute retracts.

\section{Proof of Theorem 1}
The ``only if'' part. If $\Cld_{AW}(X)$ is homeomorphic to
$\ell_2$, then $X$ is
 topologically complete by \cite{Co}. The total boundedness
of each bounded subset of $X$ follows from \cite[Theorem
5.2]{BKS}. Since $\ell_2$ is a separable absolute retract, we may
apply Theorem~\ref{th2} to conclude that the completion
$\overline{X}$ of $X$ is locally connected and contains no bounded
connected component. It remains to show that $X$ is not locally
compact at infinity. Assume the contrary, i.e., there exists a
bounded subset $B\subset X$ with locally compact complement in
$X$. Then it is easily seen that the point $\infty\in\alpha X$ has
an open neighborhood with compact closure. Whence, we can find a
compact neighborhood of $\{\infty\}$ in $\Cld_{H}(\alpha
X|\{\infty\})$. But this is impossible because of the nowhere locally
compactness of the Hilbert space $\ell_2$. This proves the ``only
if'' part of Theorem~\ref{th1}.

To prove the ``if'' part, assume that $X$ is topologically
complete, not locally compact at infinity and the completion
$\overline{X}$ of $X$ is proper, locally connected with no bounded
connected components. By Proposition~\ref{embed}, we identify
$\Cld_{AW}(X)$ with the subspace  $\Cld_H(\alpha X|\{\infty\})\setminus\{\infty\}$ of
$\Cld_H(\alpha X|\{\infty\})$. By Lemma~\ref{cubinfty}, the hyperspace
$\Cld_H(\overline{\alpha X}|\{\infty\})$ is homeomorphic to $Q$. Now consider the map
$e:\Cld_H(\alpha X|\{\infty\})\to\Cld_H(\overline{\alpha X}|\{\infty\})$
assigning to each closed subset $F\subset \alpha X$ its closure
$\overline{F}$ in $\overline{\alpha X}$ and note that this map is
an isometric embedding, which allows us to identify the hyperspace
$\Cld_{AW}(X)$ with the subspace $\{F\in\Cld_H(\overline{\alpha
X}|\{\infty\}):F=\cl(F\cap \alpha X)\}$ of $\Cld_H(\overline{\alpha
X}|\{\infty\})$. It is easy to check that this subspace is dense and
relatively complete in the Lawson semilattice
$\Cld_H(\overline{\alpha X}|\{\infty\})$. Then it is homotopically
dense in $\Cld_H(\overline{\alpha X}|\{\infty\})$ by
Proposition~\ref{complete} and Lemma~\ref{lawson}. The subset
$\Cld_{AW}(X)$, being topologically complete, is a
$G_\delta$-set in $\Cld_H(\overline{\alpha X}|\{\infty\})$.  The dense
subsemilattice $L=\Fin_H(\overline{\alpha X}|\{\infty\}) \setminus
\Fin_H(\alpha X|\{\infty\})$  is homotopy dense in
$\Cld_H(\overline{\alpha X}|\{\infty\})$, since $X$ is not locally
compact at infinity. Since $L\cap\Cld_{AW}(X)=\emptyset$, we get that $\Cld_{AW}(X)$ is a
homotopy dense $G_\delta$-subset in $\Cld_H(\overline{\alpha
X}|\{\infty\})$ with homotopy dense complement. Applying
Lemma~\ref{tool} we conclude that the space $\Cld_{AW}(X)$ is
homeomorphic to $\ell_2$.

\end{document}